\documentclass{amsart}
\usepackage[utf8]{inputenc}
\usepackage{cite}
\usepackage{amsmath,amssymb,amsfonts}
\usepackage{algorithmic}
\usepackage{graphicx}
\usepackage{textcomp}
\usepackage{xcolor}
\usepackage{lipsum}

\usepackage[foot]{amsaddr}
\usepackage{dsfont}
\newtheorem{proposition}{Proposition}

\newtheorem{theorem}{Theorem}
\newtheorem{lemma}{Lemma}

\newcommand{\sketchV}{V^{\natural}}

\title{Community Detection with a Subsampled Semidefinite Program}
\author[1]{Pedro Abdalla \and Afonso S. Bandeira}

\email{pedro.abdallateixeira@ifor.math.ethz.ch \and bandeira@math.ethz.ch}

\address{Department of Mathematics, ETH Z\"urich}

\date{February 2021}

\begin{document}

\maketitle


\begin{abstract}
Semidefinite programming is an important tool to tackle several problems in data science and signal processing, including clustering and community detection. However, semidefinite programs are often slow in practice, so speed up techniques such as sketching are often considered. In the context of community detection in the stochastic block model, Mixon and Xie~\cite{mixon2020sketching} have recently proposed a sketching framework in which a semidefinite program is solved only on a subsampled subgraph of the network, giving rise to significant computational savings. In this short paper, we provide a positive answer to a conjecture of Mixon and Xie about the statistical limits of this technique for the stochastic block model with two balanced communities.
\end{abstract}


\section{Introduction}\label{sec1}
Clustering problems are ubiquitous in data science. The main goal is to find a partition of the data into clusters in such form that the members in the same cluster are more similar than the members in different clusters. At the same time it is necessary to balance the clusters sizes to avoid the trivial solution of one cluster per member. 

A large body of work has focused on the stochastic block model, a random network model with a planted cluster structure, we refer the reader to~\cite{abbe2017survey} for a survey on recent developments.
We will focus on case of two balanced communities. Let $n$ be an even natural number and $G \sim \mathcal{G}(n;p,q)$ be a random graph on $n$ nodes drawn as follows: Randomly partition the set of $n$ vertices $V$ in two equally sized communities $V = S_1 \cup S_2$. For every pair of vertices, an edge is placed with probability $p$ if they belong to the same community $S_i$ and with probability $q<p$ otherwise, all independent. The goal is to exactly recover the partition $\{S_1,S_2\}$ from the graph alone.  Let the matrix $A \in \mathbb{R}^{n\times n}$ denote the adjacency matrix of the graph $G$. Considering a label vector $x\in\{\pm1\}^n$ representing community membership of notes\footnote{Note that there is a natural ambiguity in the labelling of each of the communities, thus the goal is best formulated in terms of recovering the partition; this corresponds of an ambiguity of global sign flip in $x$.}. The maximum likelihood estimator for the node labels $x$ is given by the program below~\cite{abbe2017survey},

\begin{equation}
\label{intractable_ML}
\begin{aligned}
\max_{x} \quad & x^{T}Ax\\
\textrm{s.t.} \quad & \mathbf{1}^{T}x = 0\\
  & x\in \{\pm 1\}^n    \\
\end{aligned}
\end{equation}
Here $\mathbf{1}$ denotes all-ones vector. 
Since it is well known that the problem (\ref{intractable_ML}) is NP-Hard \cite{garey1979computers}, we consider the standard semidefinite relaxation~\cite{goemans1995improved}. 

\begin{equation}
\label{SDP_relaxation}
\begin{aligned}
\max_{X \in \mathbb{R}^{n\times n}}      \quad & \mathrm{Tr}(AX)\\
\textrm{s.t.} \quad & X_{ii} = 1\\
  & X \succeq 0   \\
  & \mathrm{Tr}(X\mathbf{J}) = 0
\end{aligned}
\end{equation}

Where $X$ is a surrogate variable for $xx^{T}$ and $\mathbf{J}$ denotes all-ones matrix. The following theorem gives the sharp phase transition for the community detection problem with two balanced communities.

\begin{theorem}[Exact recovery threshold
\cite{Abbe-Bandeira-Hall,mossel2015consistency,bandeira2018laplacian,hajek2016achieving}]
\label{abbe_bandeira_hall}
Let $G \sim \mathcal{G}(n;p,q)$ with $p =\alpha \frac{\log n}{n}$, $q=\beta\frac{\log n}{n}$ and planted communities $\{S_1,S_2\}$. Then,
\begin{itemize}
    \item (I) For $\sqrt{\alpha}-\sqrt{\beta} <\sqrt{2}$, no algorithm can exactly recover the partition with high probability.
    
    \item (II) For $\sqrt{\alpha}-\sqrt{\beta} >\sqrt{2}$, with high probability:   
    The semidefinite program (\ref{SDP_relaxation}) has a unique solution given by $X^\natural = x^{\natural} (x^\natural) ^T$ where $x^\natural$ corresponds to the memberships of the true communities, thus achieving exact recovery.
\end{itemize}
\end{theorem}

Although polynomial time, semidefinite programs tend to be computationally costly. A powerful tool to overcome computational complexity is that of sketching (we refer the reader to \cite{woodruff2014sketching} for an instance of this idea in least squares, and \cite{tropp2017sketchy,bluhm2019dimensionality} for semidefinite optimization). In the particular framework addressed in this paper, Mixon and Xie~\cite{mixon2020sketching} have recently proposed a sketching approach wherein a potentially significantly smaller semidefinite program is solved, its size depends on the community structure strength. Our main contribution is to resolve in the positive a conjecture in~\cite{mixon2020sketching} regarding the dependency of the size of the resulting semidefinite program and the community structure strength. We now describe the sketching approach in~\cite{mixon2020sketching}, which consists of a three step process, and a tuning parameter $0< \gamma <1 $.
\begin{itemize}
    \item (Step 1) Given a graph with vertex set $V$. Subsample a smaller vertex set $\sketchV$ by sampling each node in $V$ independently at random with probability $\gamma$.
    \item (Step 2) Solve the community detection problem in the subgraph induced by $\sketchV$. 
    \item (Step 3) For each node $v$ not in $\sketchV$ use a majority vote procedure among the neighbours of $v$ in $\sketchV$ to infer its community membership.
\end{itemize}
The main goal of this paper is to determine the minimum value of $\gamma$ such that the approach above exactly recovers both communities with high probability. The computational savings come from the reduced size of the semidefinite program and so the parameter $\gamma$ governs the computational cost of the algorithm (we refer the reader to~\cite{alizadeh1995interior} for the dependency of the computational cost of semidefinite programming on the number of variables). Mixon and Xie~\cite{mixon2020sketching}  conjectured that, as long as
\begin{equation*}
   \gamma > \frac{2}{(\sqrt{\alpha}-\sqrt{\beta})^2},
\end{equation*}
the sketching approach works with high probability. Our main result provides a positive answer for this conjecture. In particular, for $\gamma =1$, we recover the threshold in Theorem \ref{abbe_bandeira_hall} (see part II).

\section{An Oracle Bound}\label{sec2}
As described above, the sketching approach consists of three steps: Sampling, solving the community detection problem for a smaller sampled graph and then recovering the entire communities using a majority vote procedure. In this section, we analyze the Step 3 and prove that it works, for a certain range of the parameter $\gamma$, as long as we know the smaller communities in Step 2. The analysis is described in the proposition below, we refer to it as an oracle bound because it assumes the knowledge of the communities in Step 2.

\begin{proposition}
\label{oracle_bound}
Let $G \sim \mathcal{G}(n;p,q)$ with planted communities $\{S_1,S_2\}$ and with $p=\alpha\frac{\log n}{n}$ and $q=\beta\frac{\log n}{n}$ satisfying $p>q$. Draw a vertex set $\sketchV$ at random by sampling each node of the graph $G$ independently at random with probability $\gamma$. Let $R_1,R_2$ be the planted communities in the sampled graph, i.e, $R_i = S_i \cap \sketchV$ for both $i \in \{1,2\}$. Moreover, let $e(v,S)$ be the number of edges of $G$ between the vertex $v$ and the set $S \subset V(G)$ where $V(G)$ is the vertex set of the graph $G$. Now, consider
\begin{equation*}
    \hat{S}_1 = R_1 \cup \{ v\in V(G) \text{\textbackslash} \sketchV: e(v,R_1) > e(v,R_{2})\} .
\end{equation*}
\begin{equation*}
    \hat{S}_2 = R_2 \cup \{ v\in V(G) \text{\textbackslash} \sketchV: e(v,R_2) > e(v,R_{1})\} .
\end{equation*}
\
Then there exists absolute constants $C,c>0$ such that, with probability $1-Cn^{-c((\alpha+\beta)\frac{\gamma}{2}-\gamma\sqrt{\alpha\beta}-1)}$, $(\hat{S}_1,\hat{S}_2)=(S_1,S_2)$. In particular, $(\hat{S}_1,\hat{S}_2)=(S_1,S_2)$ with probability $1-o(1)$, as long as
\begin{equation*}
\gamma > \frac{2}{(\sqrt{\alpha}-\sqrt{\beta})^2}.
\end{equation*}
\end{proposition}
The next lemma will play a key role in the proof of Proposition \ref{oracle_bound}, it is similar to Lemma~8 in \cite{Abbe-Bandeira-Hall} but it deals with almost balanced communities, this is crucial to our analysis. 
\begin{lemma}
\label{lemma_for_oraclebound}
Suppose $\alpha > \beta >0$. Let $X$ and $Y$ be two independent random variables with $X \sim \text{Binom}(K_1, \alpha\frac{\log n}{n})$ and $Y \sim \text{Binom}(K_2,\beta \frac{\log n}{n})$, where $K_1 = \frac{n\gamma}{2} + o(n)$ and $K_2 = \frac{n\gamma}{2} + o(n)$ as $n \rightarrow \infty$. Then,
\begin{equation*}
    \mathbb{P}(X-Y \le 0) \le n^{- ((\alpha+\beta)\frac{\gamma}{2} -\gamma\sqrt{\alpha \beta}) + o(1)}.
\end{equation*}
\end{lemma}
We present a simple and direct proof of this lemma.

\begin{proof} 
Let $\varepsilon >0$. We proceed with the Laplace transform method, for all $t\ge 0$ we write
\begin{equation}
\label{LaplaceTransorm}
    \mathbb{P}(X-Y \le 0) \le \mathbb{P}(X-Y \le \varepsilon) \le e^{t\varepsilon}\mathbb{E}e^{-t(X-Y)} := e^{-\psi(t)},
\end{equation}
where $\psi(t):= -t\varepsilon - \log \mathbb{E}e^{-t(X-Y)}$. Now we use the fact that the function $\psi(t)$ is additive for sums of independent random variables together with the formula for the moment generating function of a binomial distribution (Example 3.32 in \cite{wasserman2013all})

\begin{equation*}
  \log \mathbb{E}e^{-t(X-Y)} = K_1 \log (1 - p(1-e^{-t})) + K_2\log (1-q(1-e^{t})),
\end{equation*}
where $p = \alpha\frac{\log n}{n}$ and $q= \beta \frac{\log n}{n}$. Using the elementary inequality, $\log(1-x)\le -x$, valid for all $0\le x \le 1$, we get

\begin{equation*}
    \psi(t) \ge -\varepsilon t + K_1p(1-e^{-t}) + K_2q(1-e^{t}).
\end{equation*}
We pick $t^{*} = \log((2K_2q)^{-1}(-\varepsilon + \sqrt{\varepsilon^2 + 4K_1K_2pq}))$ in order to optimize the right hand side. The second term in the right hand side becomes

\begin{equation*}
  K_1p(1-e^{-t^{*}})= K_1p \left (1-\frac{2K_2q}{-\varepsilon + \sqrt{\varepsilon^2 + 4K_1K_2pq}}\right ).
\end{equation*}

We are interested in the behaviour of $\psi(t^{*})$ when $\varepsilon \rightarrow 0^{+}$, so we take the limit both sides in the equality above

\begin{equation*}
    \lim_{\varepsilon \rightarrow 0 ^{+}} K_1p(1-e^{-t^{*}}) = K_1p - \sqrt{K_1K_2pq}.
\end{equation*}

Similarly, we get

\begin{equation*}
    \lim_{\varepsilon \rightarrow 0 ^{+}} K_2q(1-e^{t^{*}}) = K_2q - \sqrt{K_1K_2pq}.
\end{equation*}

Now we can take the limit as $\varepsilon \rightarrow 0^{+}$ in inequality \ref{LaplaceTransorm} to obtain

\begin{equation*}
    \mathbb{P}(X-Y \le 0) \le e^{\lim_{\varepsilon \rightarrow 0 ^{+}}-\psi(t^{*})} \le e^{-(K_1p +K_2q -2\sqrt{K_1K_2pq})}.
\end{equation*}
Recall that $K_1 = \frac{n\gamma}{2} + o(n)$, $K_2 = \frac{n\gamma}{2}+o(n)$, $p = \alpha\frac{\log n}{n}$ and $q= \beta \frac{\log n}{n}$. Then,
\begin{equation*}
    \mathbb{P}(X-Y \le 0) \le e^{-\log(n)(\gamma\frac{\alpha+\beta}{2}-\gamma\sqrt{\alpha\beta}+o(1))}.
\end{equation*}
\end{proof}
We end this section with the proof of Proposition \ref{oracle_bound}.
\begin{proof}
We denote the success event by $\mathcal{E}$, i.e, the event that the communities are recovered and we condition on the event that $\sketchV$ has been drawn. By union bound we can write,

\begin{equation*}
\mathbb{P}(\mathcal{E}^{c} \mid \sketchV ) \le P_1 + P_2.
\end{equation*}
Here $P_1:= \sum_{v \in S_1} \mathds{1}_{\{v \in V(G) \text{\textbackslash}\sketchV\}}\mathbb{P}(e(v,R_1) - e(v,R_{2})\le 0)$ and $P_2$ is defined analogously. 

Observe that now the probability in the right hand side of $P_1$ is equal to 
\begin{equation*}
    \mathbb{P}\left(\sum_{j=1}^{K_1}B_j^{(p)} - \sum_{j=1}^{K_{2}}B_j^{(q)}\le 0 \right),
\end{equation*}
where $K_i = \mid R_i\mid$ and for all $j$, the random variables $B_j^{p} \sim \text{Ber}(p)$ and $B_j^{q} \sim \text{Ber}(q)$ are all independent. We set $X:= \sum_{j=1}^{K_1}B_j^{(p)} \sim \text{Binom}(K_1, \alpha\frac{\log n}{n})  $ and $Y:= \sum_{j=1}^{K_{2}}B_j^{(p)} \sim \text{Binom}(K_{2}, \beta\frac{\log n}{n})$. In order to apply Lemma \ref{lemma_for_oraclebound}, we denote the event in which both $K_1$ and $K_{2}$ lie in the interval $\frac{n\gamma}{2}(1\pm \frac{1}{\sqrt{\log n}}) $ by $\mathcal{A}$. So we can bound $P_1$ by

\begin{equation*}
P_1 \le \sum_{v \in S_1} \mathds{1}_{\{v \in V(G) \text{\textbackslash}\sketchV\}}(\mathds{1}_{\{\mathcal{A}^{c}\}}+ \mathds{1}_{\{\mathcal{A}\}}\mathbb{P}(X-Y\le 0 \mid \mathcal{A})).
\end{equation*}

We use the crude bound $\mathds{1}_{\{v \in V(G) \text{\textbackslash} \sketchV \}} \le 1$ and write

\begin{equation*}
P_1 \le \frac{n}{2}(\mathds{1}_{\{\mathcal{A}^{c}\}}+ \mathds{1}_{\{\mathcal{A}\}}\mathbb{P}(X-Y\le 0\mid\mathcal{A})).
\end{equation*}
It is easy to see that the same bound holds for $P_2$, so
\begin{equation*}
    \mathbb{P}(\mathcal{E}^{c}\mid\sketchV) \le P_1+P_2\le  n(\mathds{1}_{\{\mathcal{A}^{c}\}}+ \mathds{1}_{\{\mathcal{A}\}}\mathbb{P}(X-Y\le 0\mid\mathcal{A})).
\end{equation*}
We take the expectation with respect to $\sketchV$ both sides to obtain,
\begin{equation}
\label{bound1_Theorem1}
    \mathbb{P}(\mathcal{E}^{c})\le n(\mathbb{P}(\mathcal{A}^{c}) + \mathbb{E}_{\sketchV}\mathbb{P}(X-Y\le 0\mid\mathcal{A})).
\end{equation}
By Chernoff's small deviation inequality (Exercise 2.3.5 \cite{vershynin2018high}), there is an absolute constant $c>0$ such that
\begin{equation}
\label{bound2_Theorem1}
    \mathbb{P}(\mathcal{A}^{c})\le 2\mathbb{P}\left (\mid K_1 - \frac{n\gamma}{2}\mid> \frac{n\gamma}{2\sqrt{\log n}}\right ) \le 2e^{-c\frac{\gamma n}{\log n}} = o\left(\frac{1}{n}\right).
\end{equation}
By Lemma \ref{lemma_for_oraclebound},
\begin{equation}
\label{bound3_Theorem1}
    \mathbb{E}_{\sketchV}\mathbb{P}(X-Y\le 0\mid\mathcal{A}) \le n^{- ((\alpha+\beta)\frac{\gamma}{2} -\gamma\sqrt{\alpha \beta}) + o(1)}.
\end{equation}
By the assumption on $\gamma$, $(\alpha+\beta -2\sqrt{\alpha\beta})\frac{\gamma}{2}>1$. Therefore, there exists an $\varepsilon>0$ such that

\begin{equation*}
    \mathbb{P}(X-Y\le 0\mid\mathcal{A}) \le n^{-1 -\varepsilon +o(1)} = o\left(\frac{1}{n}\right).
\end{equation*}

Then we combine inequalities \ref{bound2_Theorem1} and \ref{bound3_Theorem1} with inequality \ref{bound1_Theorem1} to complete the proof.
\end{proof}

\section{Exact Recovery in the Subsampled Nodes}\label{sec3}
In the sampling procedure in Step 1, the unknown communities $S_1 \cap \sketchV$ and $S_2 \cap \sketchV$  are no longer guaranteed to be balanced, therefore we cannot directly use the optimization program (\ref{SDP_relaxation}) because the maximum likelihood estimator is no longer (\ref{abbe_bandeira_hall}). However, thanks to the authors in \cite{hajek2016achieving}, similar semidefinite programs can be used to handle this case. We follow the approach in \cite{hajek2016achieving}. 

To begin with, it is straightforward to see that if the communities have sizes $K$ and $n-K$, the maximum likelihood estimator becomes 
\begin{equation}
\label{intractable_ML_unbalanced}
\begin{aligned}
\max_{x} \quad & x^{T}Ax\\
\textrm{s.t.} \quad & \mathbf{1}^{T}x = (2K-n)\\
  & x\in \{\pm 1\}^n    \\
\end{aligned}
\end{equation}
Therefore we can relax the problem in the same as before, we set $X:=xx^{T}$ and write

\begin{equation}
\label{SDP_relaxation_unbalanced}
\begin{aligned}
\max_{X \in \mathbb{R}^{n\times n}}      \quad & \mathrm{Tr}(AX)\\
\textrm{s.t.} \quad & X_{ii} = 1\\
  & X \succeq 0   \\
  & \mathrm{Tr}(X\mathbf{J}) = (2K-n)^2
\end{aligned}
\end{equation}

We should remark that the formulation (\ref{SDP_relaxation_unbalanced}) requires the knowledge of the sizes of the communities. To overcome this problem, we consider a Lagrangian formulation

\begin{equation}
\label{SDP_relaxation_lagrangian}
\begin{aligned}
\max_{X \in \mathbb{R}^{n\times n}}      \quad & \mathrm{Tr}(AX) -\lambda^{*}\mathrm{Tr}(X\mathbf{J})\\
\textrm{s.t.} \quad & X_{ii} = 1\\
  & X \succeq 0   \\
\end{aligned}
\end{equation}

The intuition is that the Lagrange multiplier $\lambda^{*}$ adjusts the sizes of the communities. An important insight from \cite{hajek2016achieving} is the following: There exists a value of $\lambda^{*}$ that works for all values $K$, so the optimization program (\ref{SDP_relaxation_lagrangian}) can be used to recover unbalanced communities with unknown sizes. Indeed, the following proposition reflects it. We use the notation $G \sim \mathcal{G}(n_1,n_2,p,q)$ to denote a random graph drawn exactly in the same way as before with the exception that now the planted communities have sizes $n_1$ and $n_2$ satisfying $n_1+n_2 =n$ but $n_1$ is not necessarily equal to $n_2$.

\begin{proposition}\cite{hajek2016achieving}
\label{phase_transition_unbalanced}
Let $G \sim \mathcal{G}(K,n-K,p,q)$ with planted communities $\{S_1,S_2\}$ and with $p=\alpha\frac{\log n}{n}$ and $q=\beta\frac{\log n}{n}$ satisfying $p>q$. Then, for $\sqrt{\alpha}-\sqrt{\beta} > \sqrt{2}$, the semidefinite program (\ref{SDP_relaxation_lagrangian}) with  $\lambda^{*} = \left (\frac{\alpha -\beta}{\log \alpha -\log \beta} \right ) \frac{\log n}{n}$ exactly recovers the communities with probability $1-Cn^{-c(\frac{1}{2}(\sqrt{\alpha}-\sqrt{\beta})^2-1)}$, where $C,c>0$ are absolute constants.
\end{proposition}

\section{Main theorem}\label{sec4}
We shall proceed to the main result of this paper. We combine the ideas in sections \ref{sec2} and \ref{sec3} to establish a complete analysis of the sketching procedure. 
\begin{theorem}[Main result]
Let $G \sim \mathcal{G}(n;p,q)$ with planted communities $\{S_1,S_2\}$ and with $p=\alpha\frac{\log n}{n}$ and $q=\beta\frac{\log n}{n}$ satisfying $p>q$. Draw a vertex set $\sketchV$ at random by sampling each node of the graph $G$ independently at random with probability $\gamma$. Denote, for $i \in \{1,2\}$, $\hat{R}_i$ to be the maximum likelihood estimators of $R_i = S_i \cap \sketchV$ obtained by running the semidefinite program \ref{SDP_relaxation_lagrangian} with the input matrix $A$ being the adjacency matrix of the graph $H \subset G$ induced by $\sketchV$ and the parameter $\lambda^{\ast}$ chosen as follows: In the event that $\mid \sketchV \mid \ge 2$, set $\lambda^{\ast} = \frac{\alpha_H - \beta_H}{\log \alpha_H - \log \beta_H}\frac{\log \mid\sketchV\mid}{\mid\sketchV\mid}$, where $\alpha_{H}:= \frac{p \mid \sketchV \mid}{\log \mid \sketchV \mid}$ and $\beta_{H}:= \frac{q\mid\sketchV\mid}{\log \mid\sketchV\mid}$, otherwise set $\lambda^{\ast}=0$. Now take  
\begin{equation*}
    \hat{S}_1 = \hat{R}_1 \cup \{ v \in V(G)\text{\textbackslash} \sketchV : e(v,\hat{R}_1)>e(v,\hat{R}_{2})\}.
\end{equation*}
\begin{equation*}
    \hat{S}_2 = \hat{R}_2 \cup \{ v \in V(G)\text{\textbackslash} \sketchV : e(v,\hat{R}_2)>e(v,\hat{R}_{1})\}.
\end{equation*}
Then there exists absolute constants $C,c>0$ such that, with probability $$1-Cn^{-c((\alpha+\beta)\frac{\gamma}{2}-\gamma\sqrt{\alpha\beta}-1)},$$ $(\hat{S}_1,\hat{S}_2)=(S_1,S_2)$. In particular, with probability $1-o(1)$, $(\hat{S}_1,\hat{S}_2)=(S_1,S_2)$  as long as
\begin{equation*}
\gamma > \frac{2}{(\sqrt{\alpha}-\sqrt{\beta})^2}.
\end{equation*}
\end{theorem}
\begin{proof}
Observe that after sampling the vertex set $V(G)$ of the graph, the induced subgraph $H \subset G$ is a random graph with law $H \sim \mathcal{G}(S_1\cap \sketchV, S_2\cap \sketchV,p,q)$. 
We claim that there exists a $\lambda^{*}$ such that the optimization program \ref{SDP_relaxation_lagrangian} recovers both communities $S_1\cap \sketchV$ and $S_2\cap \sketchV$ with the desired probability. The proof of the theorem easily follows from the claim by applying Proposition \ref{oracle_bound} and union bound.

Now, we proceed to prove the claim. In order to apply Proposition \ref{phase_transition_unbalanced} we need to check that, with sufficiently large probability,

\begin{equation}
\label{last_condition}
    \sqrt{\alpha_{H}}- \sqrt{\beta_{H}} > \sqrt{2},
\end{equation}
where $\alpha_{H}:= \frac{p \mid \sketchV \mid}{\log \mid \sketchV \mid}$ and $\beta_{H}:= \frac{q\mid\sketchV\mid}{\log \mid\sketchV\mid}$ if $\mid\sketchV\mid\ge 2$ and zero otherwise. Recall, by definition, $p=\alpha \frac{\log n}{n}$ and $q=\beta \frac{\log n}{n}$. The degenerate event $\mid\sketchV\mid\le 1$ (empty set or single vertex) occurs with exponentially small probability. Indeed, observe $\mid\sketchV\mid$ is a sum of $n$ i.i.d random variables with Bernoulli distribution with mean $\gamma$, so
\smallskip

\begin{equation*}
\mathbb{P}(\mid\sketchV\mid \le 1) = (1-\gamma)^{n}+n(1-\gamma)^{n-1}\gamma \le 2e^{-\gamma(n-1)+\log n},
\end{equation*}
and
\begin{equation*}
    \mathbb{P}\left (\alpha_{H}= \alpha \frac{\mid\sketchV\mid\log n}{n\log \mid\sketchV\mid} \cap \mid\sketchV\mid\ge 2\right) = 1-2e^{-\gamma(n-1) +\log n}.
\end{equation*}
\smallskip

An analogous fact holds for $\beta_{H}$, so the event
\begin{equation*}
    \left\{\sqrt{\alpha_H}-\sqrt{\beta_{H}}= \sqrt{\frac{\mid\sketchV\mid\log n}{n\log \mid\sketchV\mid}}\left(\sqrt{\alpha}-\sqrt{\beta}\right)\right\}\cap \left \{\mid\sketchV\mid\ge 2\right\},
\end{equation*}
occurs with exponentially large probability. Since $\frac{\log n}{\log \mid\sketchV\mid}\ge 1$ (when the quotient makes sense), it is enough to prove that, with the desired probability,

\begin{equation*}
   \left (\sqrt{\frac{\mid\sketchV\mid}{n}} \right) \left(\sqrt{\alpha}-\sqrt{\beta}\right) > \sqrt{2}.
\end{equation*}
By assumption, there exists a $\delta >0$ such that $\sqrt{\alpha}-\sqrt{\beta}\ge \sqrt{\frac{2}{\gamma}}(1+\delta)$ and by the small Chernoff deviation inequality, for every $\varepsilon>0$, $\mathbb{P}\left(\frac{\mid\sketchV\mid}{n}\ge \gamma -\varepsilon \right ) \ge 1- 2e^{-c\varepsilon^2n\gamma^3}$. Putting these three facts together, we obtain, for every $\varepsilon>0$, 
\begin{equation*}
   \left (\sqrt{\frac{\mid\sketchV\mid\log n}{n\log \mid\sketchV\mid}} \right) \left(\sqrt{\alpha}-\sqrt{\beta}\right) \ge \sqrt{2}(1+\delta)\sqrt{1-\frac{\varepsilon}{\gamma}},
\end{equation*}
with exponentially large probability. We choose $\varepsilon>0$ small enough to guarantee that $(1+\delta)\sqrt{1-\frac{\varepsilon}{\gamma}} >\sqrt{1+\delta}$ and then inequality \eqref{last_condition} is satisfied with the desired probability. The claim now follows from Proposition \ref{phase_transition_unbalanced}.
\end{proof}

\section*{Acknowledgment}
The authors would like to thank Dustin Mixon, Kaiying Xie and Nikita Zhivotovsky for helpful discussions. The authors would also like to thank anonymous referees for valuable comments that improved the manuscript.


\end{document}